\documentclass[11pt,a4paper,reqno]{amsart}

\usepackage[utf8]{inputenc}
\usepackage[english]{babel}
\usepackage{amsthm}
\usepackage{amsfonts}
\usepackage{amsmath}
\usepackage{amssymb}
\usepackage{enumerate}
\usepackage{tikz-cd}

\newcommand{\terminates}{\!\!\downarrow}
\newcommand{\diverges}{\!\!\uparrow}
\newcommand{\tuple}[1]{\langle #1 \rangle}
\newcommand{\set}[1]{\{ #1 \}}
\newcommand{\compr}[2]{\{ #1 \,:\, #2 \}}
\newcommand{\simgamma}{\sim_\gamma}
\newcommand{\legammaT}{\le_{\gamma}}

\DeclareMathOperator{\PA}{PA}
\DeclareMathOperator{\Skolem}{Skolem}

\renewcommand{\le}{\leqslant}
\renewcommand{\ge}{\geqslant}
\renewcommand{\leq}{\leqslant}
\renewcommand{\geq}{\geqslant}

\newtheorem{theorem}{Theorem}[section]
\newtheorem{lemma}[theorem]{Lemma}
\newtheorem{prop}[theorem]{Proposition}
\newtheorem{corollary}[theorem]{Corollary}
\newtheorem{question}[theorem]{Question}

\theoremstyle{definition}
\newtheorem{definition}[theorem]{\bf Definition}
\newtheorem{example}[theorem]{\bf Example}
\newtheorem*{remark}{Remark}

\begin{document}

\title{Fixpoints and Relative Precompleteness}

\author[A. Golov]{Anton Golov}
\address[Anton Golov]{Radboud University Nijmegen\\
Department of Mathematics\\
P.O. Box 9010, 6500 GL Nijmegen, the Netherlands}
\email{agolov@science.ru.nl}

\author[S. A. Terwijn]{Sebastiaan A. Terwijn}
\address[Sebastiaan A. Terwijn]{Radboud University Nijmegen\\
Department of Mathematics\\
P.O. Box 9010, 6500 GL Nijmegen, the Netherlands}
\email{terwijn@math.ru.nl}

\begin{abstract}
  We study relative precompleteness in the context of the theory of
  numberings, and relate this to a notion of lowness.
  We introduce a notion of divisibility for numberings, and use it
  to show that for the class of divisible numberings, lowness and
  relative precompleteness coincide with being computable.

  We also study the complexity of Skolem functions arising from
  Arslanov's completeness criterion with parameters.
  We show that for suitably divisible numberings, these Skolem
  functions have the maximal possible Turing degree.
  In particular this holds for the standard numberings of the
  partial computable functions and the c.e.\ sets.
\end{abstract}

\keywords{fixpoint theorems, precomplete numberings}

\subjclass[2010]{% identical to 2020 for below subjects, but option [2020] does not yet work
03D25, %Recursively (computably) enumerable sets and degrees
%03D28, %Other Turing degree structures
%03B40, %Combinatory logic and lambda-calculus
03D45}  %Theory of numerations, effectively presented structures
%03D80 %Applications of computability and recursion theory

\date{\today}

\maketitle

\section{Introduction}

A numbering is a surjective function $\gamma:\omega\rightarrow S$
to a given set of objects~$S$. The computability theoretic study of
numberings was initiated by Ershov~\cite{Ershov} in a series of papers.
For a given numbering $\gamma$, we can consider two functions to be
equivalent if their values are mapped by $\gamma$ to the same objects in~$S$.
A key notion in the theory is the notion of precompleteness, which says that
every partial computable (p.c.)\ function has a total extension modulo the
numbering.
Among other things, Ershov showed that Kleene's recursion theorem~\cite{Kleene}
holds for any precomplete numbering. The recursion theorem is then the special
case for the numbering of the p.c.\ functions.
The recursion theorem has been generalized in several other ways.
For example, Arslanov~\cite{Arslanov} generalized the recursion
theorem from the computable functions to the class of functions
computable from a Turing incomplete c.e.\ set. This gives a
completeness criterion for c.e.\ sets: A c.e.\ set is Turing
complete if and only if it can compute a fixed point free function.
Selivanov~\cite{Selivanov}
proved that Arslanov's theorem also holds for any precomplete numbering.
(This result was also discussed in Barendregt and Terwijn~\cite{BarendregtTerwijn},
who considered fixed point theorems in the more general setting of
partial combinatory algebras.)
Another generalization of the recursion theorem, simultaneously generalizing
Arslanov's theorem and a theorem of Visser~\cite{Visser},  was given in
Terwijn~\cite{Terwijn2018}. This is a result about the standard
numbering of the computably enumerable (c.e.)\ sets, and it is
currently open whether this also holds for every precomplete numbering.
Variations of Arslanov's completeness criterion that arise by considering
different kinds of fixpoints were studied in Jockusch et al.~\cite{Jockuschetal}.

In this paper, we study relative precompleteness, that is, for a given oracle $A$
we consider which numberings $\gamma$ have the property that every $A$-p.c.\ function
has a total $A$-computable extension modulo $\gamma$.
We introduce a notion of lowness (Definition~\ref{def:low}), and show that for
a certain class of \emph{divisible\/} numberings that includes the standard
numberings of the p.c.\ functions and the c.e.\ sets,
lowness and $A$-precompleteness are equivalent to $A$ being computable
(Theorem~\ref{thm:W-phi-comp-equiv}).

We then proceed to study the complexity of Skolem functions that arise from
Arslanov completeness criterion with parameters.
It was shown in Terwijn~\cite{Terwijn2020} that these Skolem functions in general
cannot be computable.
Here we use the notions of lowness and divisibility from the first part of the
paper to show that they have the maximal possible Turing degree
(Corollary~\ref{cor:Skolem}).

Our notation from computability theory is mostly standard.  In the following,
$\varphi_n$ denotes the $n$-th partial computable (p.c.)\ function, in some
standard numbering of the p.c.\ functions.  Similarly, $W_n$ denotes the $n$-th
computably enumerable (c.e.)\ set, in some standard numbering of the c.e.\ sets.
We denote these numberings as $\varphi_{(-)}$ and $W_{(-)}$ respectively.
Partial computable functions are denoted by lower case Greek letters and (total)
computable functions by lower case Roman letters. $\omega$ denotes the natural
numbers.  We write $\varphi_e(n)\terminates$ if this computation is defined and
$\varphi_e(n)\diverges$ otherwise.  We let $\tuple{e, n}$ denote a computable
pairing function.  For unexplained notation we refer to
Odifreddi~\cite{Odifreddi} and Soare~\cite{Soare}.

\section{Turing degrees, up to a numbering}
\label{sec:numberings}

The theory of numberings originates from Ershov~\cite{Ershov}.

\begin{definition}
  A \emph{numbering} of a set $S$ is a surjection $\gamma : \omega \to S$.
\end{definition}

For every numbering $\gamma$ there is an associated equivalence relation
$\sim_\gamma$ on $\omega$ defined by $n \simgamma m$ iff $\gamma(n) = \gamma(m)$.
Thus, as noted in Bernardi and Sorbi~\cite{BernardiSorbi},
the study of numberings is essentially equivalent to the study of countable
equivalence relations.

Of particular interest are so-called precomplete numberings.
The notion of a precomplete numbering was introduced by Ershov~\cite{Ershov}.
We will work with the following relativized version of precompleteness.

\begin{definition}
  Given an oracle $A$, a numbering $\gamma$ is \emph{$A$-precomplete} if for
  every partial $A$-computable function $\psi$ there is a total $A$-computable
  function $f$ such that
  \[
    \psi(n)\terminates \implies \psi(n) \simgamma f(n).
  \]
  for every~$n$.
  We call $f$ a \emph{totalization of $\psi$ modulo $\gamma$}.
  Note that $\emptyset$-precomplete is the same as Ershov's notion of
  precompleteness.
\end{definition}

\begin{remark}
  A relativized notion of precompleteness was already studied by
  Selivanov~\cite{Selivanov,SelivanovSurvey}.  In his definition, Selivanov
  requires the totalization $f$ to be absolutely computable, rather than merely
  $A$-computable.
  Thus Selivanov's notion is a combination of relativization and a {\em
  lowness\/} notion. We introduce a lowness notion for numberings in
  Definition~\ref{def:low} below.
\end{remark}

\subsection{Lowness}

While our aim is to characterize the Turing degrees of particular functions, in
the context of numberings it is more natural to show results about the following
relation, which expresses a lowness property.

\begin{definition} \label{def:low}
  Given a numbering $\gamma$ and oracles $A$ and $B$, we say that $A$ is
  \emph{$(\gamma, B)$-low} if for every $A$-computable function $f$ there is a
  $B$-computable function $g$ such that for all $n$,
  \begin{equation} \label{eq:lift}
    f(n) \sim_\gamma g(n).
  \end{equation}
\end{definition}

We will denote this by $A \legammaT B$.  We will say that $A$ is
\emph{$\gamma$-low} if $A \legammaT \emptyset$.  A function $g$ satisfying
\eqref{eq:lift} will be called a \emph{$\gamma$-lift of $f$ to $B$}.

Let us now look at some basic properties and some simple examples of lowness
results.

\begin{theorem}
  For every $\gamma$, the relation $\legammaT$ is a preorder.
\end{theorem}

\begin{proof}
  Reflexivity is trivial.  For transitivity, suppose $A \legammaT B \legammaT C$
  and let $f$ be an $A$-computable function.  Since $A \legammaT B$, $f$ has a
  $\gamma$-lift $g$ to $B$, and since $B \legammaT C$, $g$ has a $\gamma$-lift
  $h$ to $C$.  Since these are $\gamma$-lifts, for every $n$ we have that
  \[
    f(n) \simgamma g(n) \simgamma h(n).\qedhere
  \]
\end{proof}

The following result shows that the lowness relation $\legammaT$
is coarser than Turing reducibility.

\begin{prop} \label{prop:leT-legammaT}
  For every numbering $\gamma$ and all $A$ and $B$, if $A \le_T B$ then $A
  \legammaT B$.
\end{prop}

\begin{proof}
  Since by assumption every $A$-computable function is $B$-computable, for any
  $A$-computable $f$ we can take $f$ itself as its $\gamma$-lift to $B$.
\end{proof}

The converse of Proposition~\ref{prop:leT-legammaT} does not hold by
Example~\ref{ex:low}, since for the numbering there
$A\legammaT \emptyset$ for every $A$, but not every $A$ is computable.
In Theorem~\ref{thm:legammaT-div-leT} we prove a partial converse of
Proposition~\ref{prop:leT-legammaT}.

\begin{corollary}
  A set $A$ is $\gamma$-low iff it is minimal with respect to the $\legammaT$
  ordering.
\end{corollary}

\begin{proof}
  $A$ is $\gamma$-low if $A\leq_\gamma \emptyset$. Since $\emptyset \leq_T B$
  for all $B$, it follows from Proposition~\ref{prop:leT-legammaT} that $A\leq_\gamma
  B$ for all $B$.
\end{proof}

For numberings $\gamma_1$ and $\gamma_2$, we say that $\gamma_2$ is a
\emph{quotient\/} of $\gamma_1$ if $n \sim_{\gamma_1} m$ implies that $n
\sim_{\gamma_2} m$ for all $n$ and~$m$.

\begin{prop} \label{lem:low-quotient}
  Let $\gamma_1$ and $\gamma_2$ be numberings and let $\gamma_2$ be a quotient
  of $\gamma_1$.  If $A \le_{\gamma_1} B$ then $A \le_{\gamma_2} B$.
\end{prop}

\begin{proof}
  Let $f$ be $A$-computable and let $g$ be its $\gamma_1$-lift to $B$. For every
  $n$, $f(n) \sim_{\gamma_1} g(n)$.  Since $\gamma_2$ is a quotient of
  $\gamma_1$, it follows that $f(n) \sim_{\gamma_2} g(n)$.
\end{proof}

\begin{example} \label{ex:low}
  Let $A$ be an arbitrary oracle.  Recall that $\varphi^A_{(-)}$ is the standard
  numbering of partial $A$-computable functions.  By the $S$-$m$-$n$-theorem,
  for every $A$-computable function $f$ there exists a primitive recursive
  function $g$ such that for all $n$,
  \[
    \varphi^A_{f(n)} = \varphi^A_{g(n)}
  \]
  and hence $A$ is $\varphi^A_{(-)}$-low.

  Since the numbering of $A$-c.e.\ sets $W^A_{(-)}$ is a quotient of
  $\varphi^A_{(-)}$, every $A$ is $W^A_{(-)}$-low by
  Proposition~\ref{lem:low-quotient}.
\end{example}

\subsection{Divisibility}

In order to turn results about lowness into results about Turing reducibility,
we will make use of the following notion of \emph{divisibility}.\footnote{
The notion of divisibility is very similar to the notion of an $A$-wide numbering
from Selivanov~\cite{SelivanovWide}.
  A numbering $\gamma$ is $A$-wide if there exists an $A$-c.e.\ sequence of
  $\gamma$-index sets $(S_n)_{n \in \omega}$ and an $A$-computable function
  $u$ such that for all $n$, $u(n) \in S_n - \bigcup_{k \neq n} S_k$.
Every $A$-wide numbering is $(\omega, A)$-divisible, but it is not clear whether
the converse holds.}
For $x\in\omega$, we let $[x]_{\simgamma}$ denote the equivalence class of $x$
under $\sim_\gamma$.

\begin{definition} \label{def:divisible}
  Let $n \le \omega$.  A numbering $\gamma$ is said to be \emph{$(n,
  A)$-divisible} if there exist $A$-computable sequences $(x_i)_{i < n}$ and
  $(e_i)_{i < n}$ such that:
  \begin{enumerate}[\rm (i)]
    \item For all $i < n$, $[x_i]_{\simgamma} \subseteq W^A_{e_i}$.
    \item For all distinct $i, j < n$, $[x_i]_{\simgamma} \cap W^A_{e_j} =
      \emptyset$.
  \end{enumerate}
  We say that the points $x_i$ and the sets $W^A_{e_i}$ \emph{witness} the
  divisibility of $\gamma$.
We use $n$-divisible to mean $(n, \emptyset)$-divisible.\footnote{\label{fn}
For a c.e.\ equivalence relation (or ceer, cf.\ section~\ref{sec:ceers} below)
that is $\omega$-divisible,  
since the equivalence classes $[x_i]_\sim$ are c.e.,  
we can simply take $W_{e_i}=[x_i]_\sim$ in Definition~\ref{def:divisible}. 
In Andrews and Sorbi~\cite{AndrewsSorbi2019}, a ceer was called \emph{light\/} 
if it has an effectively enumerable antichain. 
So for ceers, $\omega$-divisible is exactly the same as light.}
\end{definition}

\begin{prop} \label{lem:div-quotient}
  Let $\gamma_1$ and $\gamma_2$ be numberings and let $\gamma_2$ be a quotient
  of $\gamma_1$.  If $\gamma_2$ is $(n, A)$-divisible then so is $\gamma_1$.
\end{prop}

\begin{proof}
  The same witnesses can be used and the requirements are easily verified.
\end{proof}

\begin{example} \label{ex:divisible}
  For an arbitrary oracle $A$, the numbering $W^A_{(-)}$ is $(\omega,
  A)$-divisible since we can take the enumerations $W^A_{x_i} = \set{i}$ and
  \[
      W^A_{e_i} = \compr{d \in \omega}{i \in W^A_d}.
  \]

  It is then easily verified that the conditions from
  Definition~\ref{def:divisible} are met, because when $i \neq j$ then $j\notin
  W^A_{x_i}$.  Using the fact that $W^A_{(-)}$ is a quotient of
  $\varphi^A_{(-)}$, we also see that $\varphi^A_{(-)}$ is $(\omega,
  A)$-divisible by Proposition~\ref{lem:div-quotient}.
\end{example}

\begin{prop}
  \label{lem:leT-lifts-divisibility}
  For any oracles $A \le_T B$ and any numbering $\gamma$, if $\gamma$ is $(n,
  A)$-divisible then it is also $(n, B)$-divisible.
\end{prop}

\begin{proof}
  Trivial, since the same witnesses can be used.
\end{proof}

The following result is a partial converse of Proposition~\ref{prop:leT-legammaT}.

\begin{theorem}
  \label{thm:legammaT-div-leT}
  Let $\gamma$ be a numbering and let $A \legammaT B$.  If $\gamma$ is $(2,
  B)$-divisible then $A \le_T B$.
\end{theorem}

\begin{proof}
  Let $\gamma$ be such a numbering and let the points $x, y$ and the c.e.\ sets
  $X, Y$ witness the divisibility of $\gamma$.
  Define $A$-computable $f$ by
  \[
    f(n) =
    \begin{cases}
      x & \text{if $n \not \in A$}\\
      y & \text{if $n \in A$}
    \end{cases}.
  \]

  Let $g$ be a $\gamma$-lift of $f$ to $B$.  Note that by property (i) above,
  $g(n) \in X \cup Y$, and by property (ii), if $g(n) \sim_\gamma x$ then $g(n)
  \not \in Y$ and vice-versa.  Define $h$ as follows:
  \[
    h(n) =
    \begin{cases}
      0 & \text{if $g(n) \in X$}\\
      1 & \text{if $g(n) \in Y$}\\
    \end{cases}.
  \]

  Since $h$ is $B$-computable and $h = \chi_A$, $A$ is $B$-computable.
\end{proof}

\subsection{Partial computable functions and c.e.\ sets}

Before we continue, let us consider a few more examples and investigate the
consequences of Theorem~\ref{thm:legammaT-div-leT} for the standard numberings
$\varphi^A_{(-)}$ and $W^A_{(-)}$.  The contrapositive of
Theorem~\ref{thm:legammaT-div-leT} gives us some examples of non-divisible
numberings.

\begin{corollary}
  For noncomputable $A$, the numberings $W^A_{(-)}$ and $\varphi^A_{(-)}$ are
  not $2$-divisible.
\end{corollary}

\begin{proof}
  Immediate by the theorem together with Example~\ref{ex:low}.
\end{proof}

The following propositions will serve to characterize the behavior of
$\legammaT$ for the numbering of the $A$-c.e.\ sets for any $A$.

\begin{prop} \label{lem:lower}
  If $A$ is $\gamma$-low and $B \le_T A$, then $B$ is $\gamma$-low as well.
\end{prop}
\begin{proof}
  Since $B \le_T A$ implies $B \legammaT A$ by Proposition~\ref{prop:leT-legammaT},
  we have $B \legammaT A \legammaT \emptyset$.
\end{proof}

\begin{prop} \label{lem:upper}
  Suppose $\gamma$ is $(2, A)$-divisible and $A \le_T B$.
  If $C \legammaT B$, then $C \le_T B$.
\end{prop}
\begin{proof}
  Since $A \le_T B$, $\gamma$ is also $(2, B)$-divisible
  by Proposition~\ref{lem:leT-lifts-divisibility}.  It now follows from
  Theorem~\ref{thm:legammaT-div-leT} that $C \le_T B$.
\end{proof}

For the numbering $\gamma$ of the $A$-c.e.\ sets, we can now characterize the
relation $\leq_\gamma$ on the upper and lower cone of $A$ as follows:

\begin{theorem}
  Fix $A$, and let $\gamma$ be one of $W^A_{(-)}$ and $\varphi^A_{(-)}$.
  Then
  \begin{enumerate}[\rm (i)]
    \item every $B\leq_T A$ is $\gamma$-low,

    \item for all $B,C\geq_T A$, we have $B\leq_\gamma C$ if and only if
      $B\leq_T C$.

  \end{enumerate}
\end{theorem}

\begin{proof}
  By Example~\ref{ex:low}, $A$ is $\gamma$-low, so item (i) immediately
  follows from Proposition~\ref{lem:lower}.

  For item (ii), as we have seen in Example~\ref{ex:divisible}, $\gamma$ is
  $(\omega, A)$-divisible, hence certainly $(2, A)$-divisible.
  Suppose that $B, C \geq_T A$.
  If $B\leq_\gamma C$ then $B\leq_T C$ by Proposition~\ref{lem:upper}.
  The converse always holds by Proposition~\ref{prop:leT-legammaT}.
\end{proof}

\subsection{Ceers and c.e.\ numberings} \label{sec:ceers}

Every numbering $\gamma$ gives rise to the equivalence relation $\simgamma$.
A particularly interesting class of numberings are those for
which $\simgamma$ is a computably enumerable equivalence relation,
typically called a \emph{ceer\/}.
These were already studied by Ershov~\cite{Ershov}, who called
them \emph{positive\/} numberings.
Every ceer, in turn, gives rise to a c.e.\ numbering by sending
$n \in \omega$ to its equivalence class.
Ceers have been studied extensively in recent years, see for example
Andrews, Badaev, and Sorbi~\cite{AndrewsBadaevSorbi},
Andrews et al.~\cite{Andrewsetal},
Andrews and Sorbi~\cite{AndrewsSorbi2018, AndrewsSorbi2019},
and the references mentioned in these papers.

\begin{definition}
  A numbering $\gamma$ is \emph{c.e.\/} if the set
  \[
    \compr{ \tuple{x, y} }{ x, y \in \omega, x \simgamma y }
  \]
  is computably enumerable.
\end{definition}

Since each equivalence class of a c.e.\ numbering is c.e., the notion of
$n$-divisibility is trivial for finite $n$ provided that there are sufficiently
many equivalence classes.  However, this does not extend to the infinitary case,
since it may not be possible to enumerate an infinite subset of the equivalence
classes without repetition.

\begin{prop}
  If $\gamma$ is a c.e.\ numbering with at least $n \in \omega$ equivalence
  classes then $\gamma$ is $n$-divisible.
\end{prop}

\begin{proof}
  Select $n$ non-$\gamma$-equivalent elements $x_i$ and note that for all $0 \le
  i < n$, the set
  \[
    \compr{ y \in \omega }{ x_i \simgamma y }
  \]
  is computably enumerable.
\end{proof}

\begin{corollary}
  For every c.e.\ numbering $\gamma$ with at least two equivalence classes, the notions $A
  \le_T B$ and $A \legammaT B$ are equivalent.
\end{corollary}

\begin{proof}
  This follows immediately from Proposition~\ref{prop:leT-legammaT}
  and~\ref{thm:legammaT-div-leT}.
\end{proof}

\begin{theorem}
  There exists a c.e.\ numbering $\gamma$ with $\omega$ equivalence classes
  that is not $\omega$-divisible.\footnote{Note added in proof:
  in the terminology of Andrews and Sorbi\cite{AndrewsSorbi2019},
  this says that there exists a ceer that is not light.
  This theorem is thus also a corollary of Theorem 3.3 from~\cite{AndrewsSorbi2019}}
\end{theorem}

In Proposition~\ref{thm:ce-precomplete-divisible} we will see that such a $\gamma$
cannot be precomplete.

\begin{proof}
  We construct a c.e.\ set of pairs $A$ and let $\simgamma$ be the least
  equivalence relation containing $A$.  The numbering $\gamma$ then arises by
  sending every $n \in \omega$ to its $\simgamma$-equivalence class.

  To guarantee that $\gamma$ is not $\omega$-divisible, we ensure that for every
  total computable function $f$ there is some equivalence class of $\gamma$ that
  appears in the output of $f$ at least twice, and thus that taking $x_i = f(i)$
  does not satisfy requirement (ii) in Definition~\ref{def:divisible}.
  We do this by ensuring the following requirements are satisfied:
  \[
    R_e : \text{$f = \varphi_e$ total and injective}
      \implies \exists i, j.\, \tuple{f(i), f(j)} \in A.
  \]

  Let $A_s$ be the set of pairs enumerated up to stage $s$.
  Throughout the construction, we keep track of a function $g_s$ that maps
  $n$ to the least requirement $e$ that has contributed a value to the
  equivalence class of $n$ in $A_s$, or $\infty$ if there is no such
  requirement.
  We also track which requirements have acted.

  At stage $s = 0$, set $A_0$ to be empty.

  At stage $s+1$, we say that a requirement $R_e$ needs attention if it has not
  acted previously and there exist $i, j$ such that $x = \varphi_{e, s}(i)\terminates$
  and $y = \varphi_{e, s}(j)\terminates$ and $g_s(x), g_s(y) > e$.

  If no requirement $R_e$, with $e < s$, requires attention, do nothing for this
  stage.  Otherwise, let $e < s$ be least such that $R_e$ requires attention,
  enumerate $\tuple{x, y}$ into $A$, and mark $R_e$ as having acted.  This
  completes the construction.

  To see that the requirements are satisfied, it suffices to show that if $f =
  \varphi_e$ is total and injective, then $R_e$ eventually acts, since any
  requirement that acts is clearly satisfied.  Suppose the antecedent of $R_e$
  is satisfied.  By induction, there is some stage $s$ after which no
  requirement $d < e$ acts and thus $g^{-1}_s(\set{0, \ldots, e-1})$ is fixed.
  Since $f$ is total and injective it is also unbounded, and thus there
  exist $i, j$ such that $g_s(f(i)), g_s(f(j)) > e$ and $R_e$ will thus
  eventually act.

  It remains to show that $A$ induces infinitely many equivalence classes in
  $\gamma$, which we will do by showing every equivalence class is finite.  Let
  $x \in \omega$.  If the equivalence class of $x$ is a singleton, then it is
  clearly finite.  Otherwise, there is some requirement $R_e$ that first
  extended the equivalence class of $x$.  By construction, every next extension
  was by some requirement $R_d$ with $d < e$, and since there are only finitely
  many such requirements, the equivalence class of $x$ was only extended
  finitely often and is thus finite.

  It follows that every equivalence class is finite, and thus that $\simgamma$
  has infinitely many equivalence classes.  However, since every computable
  function is non-injective modulo $\gamma$, it follows that $\gamma$ is not
  $\omega$-divisible.
\end{proof}

In practice, many ceers from the literature are in fact $\omega$-divisible.  In
particular, we will see in Proposition~\ref{thm:ce-precomplete-divisible} that all
precomplete ceers are $\omega$-divisible.
First, let us consider two examples from Visser~\cite{Visser}.

\begin{example} \label{ex:lambda}
  Let $M_{\lambda \beta}$ be the set of $\lambda$-calculus terms modulo
  $\beta$-equivalence and let $\gamma_\lambda$ be the numbering given by some effective
  coding of terms.

  Since $\beta$-equivalence is a c.e.\ relation, $\gamma_\lambda$ is c.e.
  To see that $\gamma_\lambda$ is $\omega$-divisible, note that by strong
  normalization and the diamond property, it suffices to enumerate any infinite
  set of non-equal terms in normal form.
  This can be done, for example, by taking the Church numerals
  $\lambda f.\lambda x. f^n(x)$.
\end{example}

\begin{example}
  Let $A$ be an arithmetical set and let $\gamma$ be the numbering induced by the
  equivalence relation
  \[
    n \simgamma m \iff \PA \vdash \forall x.\, \varphi^A_n(x) \simeq \varphi^A_m(x).
  \]

  Since we can enumerate proofs in $\PA$, $\gamma$ is a c.e.\ numbering.
  To see that $\gamma$ is $\omega$-divisible we can take $\varphi^A_{f(n)} =
  n$, which gives us an enumeration $f$ of codes of distinct functions.
  Since $\PA$ is sound, these codes are not identified by $\gamma$.

  Note that by Theorem~\ref{thm:legammaT-div-leT}, it follows that if $A$
  noncomputable then $A$ is not $\gamma$-low.
  This is in contrast with the numbering $\varphi^A_{(-)}$, where $A$ is always
  $\varphi^A_{(-)}$-low.
\end{example}

An analogue of the following result for wide numberings was already mentioned
by Selivanov~\cite{SelivanovSurvey}.

\begin{prop} \label{thm:ce-precomplete-divisible}
  Every precomplete c.e.\ numbering is $\omega$-divisible.\footnote{
An alternative way to phrase this is to say that every precomplete ceer is light, 
cf.\ footnote~\ref{fn}. 
This is also easy to see using the result of 
Bernardi and Sorbi~\cite{BernardiSorbi} that every precomplete ceer is 
universal, so that in particular the identity relation reduces to it. 
The latter property is equivalent to being light.}
\end{prop}

\begin{proof}
  This follows from a theorem due to Lachlan~\cite{Lachlan}, stating that every
  two precomplete c.e.\ numberings are recursively isomorphic.
  Let $\gamma$ be a precomplete c.e.\ numbering.
  By Lachlan's theorem $\gamma$ is recursively isomorphic to $\gamma_\lambda$
  from Example~\ref{ex:lambda} above.
  There thus exists a recursive permutation $h$ such that $\gamma_\lambda(n) =
  \gamma_\lambda(m)$ iff $\gamma(n) = \gamma(m)$.
  Taking $(x_n)_{n \in \omega}$ to be the enumeration of Church numerals,
  $(h(x_n))_{n \in \omega}$ witnesses that $\gamma$ is $\omega$-divisible as
  well.
\end{proof}

\section{Relative precompleteness}

The following is Ershov's version of the recursion theorem for precomplete
numberings~\cite{Ershov2}. In relativized form it reads as follows:

\begin{theorem}[Recursion theorem for precomplete numberings]
  If $\gamma$ is an $A$-precomplete numbering then every $A$-computable function
  $f$ has a fixpoint modulo $\gamma$; that is, for every $A$-computable $f$
  there is an $n \in \omega$ such that $f(n) \simgamma n$.
\end{theorem}

Kleene's second recursion theorem (cf.\ Moschovakis~\cite{Moschovakis} for the
rich history of this result) is a uniform, parameterized version of the first
recursion theorem.
It also holds for precomplete numberings (with the same proof), and reads as
follows:

\begin{theorem}[Recursion theorem with parameters] \label{recthmparam}
  If $\gamma$ is an $A$-pre\-com\-plete numbering then for every binary
  $A$-computable function $h$ there is a unary $A$-computable function $f$ such
  that for all $n \in \omega$,
  \[
    h(n, f(n)) \simgamma f(n).
  \]
\end{theorem}

Arslanov~\cite{Arslanov} generalized the recursion theorem from computable functions
to the class of functions that are computable from an incomplete c.e.\ set.
Phrased differently: A c.e.\ set is Turing complete if and only if it can compute
a fixpoint free function.
This is Arslanov's completeness criterion.
Selivanov~\cite{Selivanov} (see also \cite{BarendregtTerwijn})
proved that Arslanov's theorem also holds for any precomplete numbering:

\begin{theorem}[Arslanov's completeness criterion for precomplete numberings]
  \label{thm:arslanov}
  If $A$ is Turing incomplete and c.e.\ and $\gamma$ is a precomplete numbering,
  then every $A$-computable function $f$ has a fixpoint modulo $\gamma$.
\end{theorem}

At first glance, we may hope to prove Arslanov's completeness criterion by
showing that that if $A$ is incomplete and c.e.\ and $\gamma$ is precomplete,
then $\gamma$ is also $A$-precomplete.
In the remainder of this section we show that this is not the case for
precomplete $\omega$-divisible numberings.
As we have seen in the previous section, this includes the numberings $W_{(-)}$
and $\varphi_{(-)}$.

\begin{lemma}
  \label{lem:ucomp-low}
  If an $A$-precomplete numbering $\gamma$ is $2$-divisible and the totalization
  $u$ of the universal function modulo $\gamma$ is computable, then $A$ is
  $\gamma$-low and therefore computable.
\end{lemma}
\begin{proof}
  Let $\gamma$ be such a numbering and let $u$ be the totalization of the
  universal function.  Suppose $u$ is computable.  Let $f = \varphi^A_e$ be a
  total $A$-computable function.
  Define a computable function $g$ by
  \[
    g(n) = u(e, n).
  \]

  Now for all $n$,
  \[
    g(n) = u(e, n) \simgamma \varphi^A_e(n) = f(n)
  \]
  and thus $A$ is $\gamma$-low.  By Theorem~\ref{thm:legammaT-div-leT}, $A$ is
  computable.
\end{proof}

\begin{theorem}
  \label{thm:omegasep-Aprecomplete-Acomp}
  Let $A \le_T B$.
  If a numbering $\gamma$ is $(\omega, A)$-divisible and $\gamma$ is
  $B$-precomplete, then $A \equiv_T B$.
\end{theorem}

\begin{proof}
  Let $A \le_T B$ and let $u$ be the totalization of a partial $B$-computable
  universal function modulo $\gamma$.  Let $(x_i)_{i \in \omega}$ and $(X_i)_{i
  \in \omega}$ witness that $\gamma$ is $(\omega, A)$-divisible.

  Let $U(e, x)$ be the use of $u(e, x)$ and define a $B$-computable function $f$
  by
  \[
    f(n) = \max_{\substack{i \le n\\ j \le n}} U(i, j).
  \]

  If $\lim_{n \to \infty} f(n)$ is finite then $u$ is computable and thus $B$ is
  $\gamma$-low by Lemma~\ref{lem:ucomp-low} and is thus computable by
  Theorem~\ref{thm:legammaT-div-leT}.  It follows that $A \equiv_T B \equiv_T
  \emptyset$.

  Suppose, then, that $\lim_{n \to \infty} f(n)$ diverges.  Let $D_{(-)}$ be the
  canonical numbering of finite sets.
  Since $f$ is $B$-computable, there is a $B$-computable function $g$ such that
  \[
    D_{g(n)} = \compr{ b \in B }{b \le f(n+1)}.
  \]

  Let $u'(e, n, m)$ be the computable function that acts like $u(e, n)$ but with
  queries to the oracle $B$ replaced by reads from $D_m$.

  Now fix $e$ to be a code of $n \mapsto x_{g(n)}$.
  Note that this function is $B$-computable, since $g$ is $B$-computable, the
  sequence $x_i$ is $A$-computable, and $A\le_T B$.
  Note that we have
  \begin{equation} \label{ug}
    u(e,n) \sim_\gamma \varphi^B_e(n) = x_{g(n)}.
  \end{equation}
  Define the partial $A$-computable function $\delta$ by
  \[
    \delta(n, m) = \text{least $i$ such that $u'(e, n, m) \in X_i$}.
  \]

  For any $n, m \in \omega$ with $n\geq e$, if $D_m$ agrees with $B$ on the
  first $f(n)$ bits, then $\delta(n, m)$ terminates after using at most the
  first $f(n)$ bits of $D_m$ and $D_{\delta(n, m)}$ agrees with $B$ on the first
  $f(n+1)$ bits.
  Namely, $u'(e,n,m) = u(e,n)$ since by assumption $D_m$ gives the correct
  answers to the oracle queries below $f(n)$, which includes the use of
  $u(e,n)$, and because by \eqref{ug} we have $u(e,n) \sim_\gamma x_{g(n)}$
  it follows that $u'(e, n, m) \in X_{g(n)}$ by the properties of the dividing
  sets $X_i$.
  Hence $\delta(n,m)\terminates = g(n)$.

  Let $D_c$ be the code of the first $f(e)$ bits of the oracle~$B$.
  By recursion, we can define
  \begin{align*}
    b(0) &= D_c\\
    b(k+1) &= \delta(k, b(k)).
  \end{align*}

  Now $b$ is an $A$-computable sequence of canonical codes of increasing initial
  segments of $B$, where $D_{b(n)}$ contains at least $f(n)$ bits if $n \ge e$.
  Since $\lim_{n \to \infty} f(n) = \infty$, eventually we can compute all of
  $B$ this way. Since $\delta$ is $A$-p.c., it follows that $B\leq_T A$.
\end{proof}

Putting this together with earlier results, we obtain the following equivalence.

\begin{theorem}
  \label{thm:W-phi-comp-equiv}
  For any precomplete $\omega$-divisible numbering $\gamma$, the following are equivalent:
  \begin{enumerate}
    \item $A$ is computable.
    \item $A$ is $\gamma$-low.
    \item $\gamma$ is $A$-precomplete.
  \end{enumerate}
  In particular, this is the case for $\gamma$ one of $W_{(-)}$ and $\varphi_{(-)}$,
\end{theorem}

\begin{proof}
  If $A$ is computable, by Proposition~\ref{prop:leT-legammaT} it is $\gamma$-low.
  Since $\gamma$ is already precomplete, it is also $A$-precomplete.

  If $A$ is $\gamma$-low, then by Theorem~\ref{thm:legammaT-div-leT} it is in
  fact computable.

  If $\gamma$ is $A$-precomplete, Theorem~\ref{thm:omegasep-Aprecomplete-Acomp}
  applies and again $A$ is computable.

  As we have seen $W_{(-)}$ and $\varphi_{(-)}$ are precomplete and
  $\omega$-divisible and thus satisfy the requirements of the theorem.
\end{proof}

Taking $A$ noncomputable, and considering the numbering $W_{(-)}$, we see from
Theorem~\ref{thm:W-phi-comp-equiv} that this numbering is precomplete, but not
$A$-pre\-com\-plete.
If moreover $A$ is c.e. and incomplete, then by Arslanov's completeness
criterion (Theorem~\ref{thm:arslanov}) we know that every $A$-computable
function has a fixpoint for this numbering.
So we see that the existence of fixpoints of $A$-computable functions does not,
in general, give us $A$-precompleteness, even for a precomplete numbering.

\begin{corollary} \label{cor:precomplete}
  There exists a set $A$ and a numbering $\gamma$ such that $\gamma$ is
  precomplete but not $A$-precomplete.
\end{corollary}

We currently do not know if the converse implication also does not hold,
so we ask:

\begin{question} \label{Q}
  Does there exists a set $A$ and a numbering $\gamma$ such that $\gamma$ is
  $A$-precomplete but not precomplete?
\end{question}

\section{Skolem functions}

Just as the recursion theorem has a version with parameters
(Theorem~\ref{recthmparam}), we can formulate in the same way a parameterized
version of
Arslanov's completeness criterion (Theorem~\ref{thm:arslanov}).
Again we can formulate this for arbitrary precomplete numberings, and in
relativized form, as follows:

\begin{theorem}[Arslanov's completeness criterion with parameters,
for precomplete numberings]
  \label{thm:arslanovparam}
  If $A$ is a incomplete c.e.\ set and $\gamma$ is a precomplete numbering then
  for every $A$-computable binary function $h$ there is an $A$-computable
  function $f$ such that
  \begin{equation} \label{Skolem}
    h(n, f(n)) \simgamma f(n).
  \end{equation}
\end{theorem}

\subsection{Complexity of Skolem functions}

It was already stated in Arslanov~\cite{Arslanov} that this form of Arslanov's
completeness criterion holds for the standard numbering of c.e.\ sets $W_{(-)}$.
Following the notation in \cite{Terwijn2020}, we refer to $f$ as in \eqref{Skolem}
as the Skolem function, since it is the Skolemization of
\[
  \forall n\exists x.\, h(n, x) \sim_\gamma x,
\]
which holds by Theorem~\ref{thm:arslanov}.

\begin{definition}
  For a given numbering $\gamma$, we say that there exist
  $B$-computable Skolem functions for
  $A$-computable functions if $f$ as in \eqref{Skolem} can be
  chosen to be $B$-computable. If this is the case we write
  $\Skolem(A,B,\gamma)$.
\end{definition}

Note that Theorem~\ref{thm:arslanovparam} says that $\Skolem(A, A, \gamma)$
for every incomplete c.e.\ $A$ and precomplete~$\gamma$.
This can be proved by analyzing the original proof of Arslanov.
For completeness we include a proof here.

\begin{proof}[Proof of Theorem~\ref{thm:arslanovparam}]
  Let $A$ be a incomplete, c.e.\ set and let $\gamma$ be a precomplete
  numbering.
  Let $\hat h$ be a computable approximation of $h$ and let $m$ be its
  $A$-computable modulus.  By the properties of a modulus we have that for all
  $n, x \in \omega$ and all $s \ge m(n, x)$, $\hat h(n, x, s) = h(n, x)$.

  Define
  \[
    \psi(n, x, k) =
    \begin{cases}
      \hat h(n, x, s_k) & \text{if $k \in \emptyset'$ and $s_k$ is minimal such that $k \in \emptyset_s'$}\\
      \uparrow & \text{if $k \not \in \emptyset'$}.
    \end{cases}
  \]

  Let $g$ be a totalization of $\psi$ modulo $\gamma$.  By the recursion theorem
  with parameters there is a computable $\hat f(n, k)$ such that for all $n, k
  \in \omega$,
  \[
    \hat f(n, k) \sim_\gamma g(n, \hat f(n, k), k).
  \]

  Define
  \[
    r(n) = \text{$k$ where $\tuple{s, k} = \mu \tuple{s, k} [ m(n, \hat f(n, k))
    \le s \wedge k \in \emptyset_s' ]$}.
  \]

  We claim $r$ is total.  Note that by minimality, $s = s_k$ from above.
  Suppose there exists $n$ such that no such $s$ and $k$ exist.  Then for all $k
  \in \emptyset'$, $m(n, \hat f(n, k)) > s_k$ and thus $k \in \emptyset_{m(n,
  \hat f(n, k))}'$.  It follows that $\emptyset' \le_T A$, contradicting our
  assumption that $A$ is incomplete.

  Since $r(n)$ is total and $m(n, \hat f(n, r(n))) \le s_{r(n)}$ it follows that
  \[
    h(n, \hat f(n, r(n)) =  \hat h(n, \hat f(n, r(n)), s_{r(n)})
  \]
  and thus since $r(n) \in \emptyset'$
  \begin{align*}
    h(n, \hat f(n, r(n))
    &= \hat h(n, \hat f(n, r(n)), s_{r(n)})\\
    &\sim_\gamma g(n, \hat f(n, r(n)), r(n))\\
    &\sim_\gamma \hat f(n, r(n)).
  \end{align*}

  Therefore, $f(n) = \hat f(n, r(n))$ is a fixpoint of $h$ modulo $\gamma$.
\end{proof}

This shows that the Skolem functions have degree at most $A$.  In
Terwijn~\cite[Theorem 3.1]{Terwijn2020} it was shown that, for the numbering
$W_{(-)}$, the Skolem functions in general cannot be computable.  We will now
show that the degree of the Skolem functions for an incomplete c.e.\ set $A$ is
in fact equal to~$A$.

\begin{theorem}
  \label{thm:skolem-low}
  For every $A, B$ and every numbering $\gamma$, if $\Skolem(A,B,\gamma)$
  then $A \legammaT B$.
\end{theorem}

\begin{proof}
  Suppose that for every $A$-computable $h$ there exists such a
  $B$-com\-pu\-table~$f$.
  Let $g$ be an $A$-computable function and define $h(n, x) = g(n)$.
  Take $f$ $B$-computable such that $h(n, f(n)) \sim_\gamma f(n)$ for all $n$.
  It now follows that
  \[
    g(n) = h(n, f(n)) \sim_\gamma f(n).\qedhere
  \]
\end{proof}

\begin{corollary} \label{cor:divSkolem}
  Let $A$ and $B$ be oracles, and let $\gamma$ be a $(2, B)$-divisible
  numbering.
  If $\Skolem(A,B,\gamma)$ then $A\le_T B$.
\end{corollary}

\begin{proof}
  It follows from Theorem~\ref{thm:skolem-low} that $A \legammaT B$, and
  thus $A\leq_T B$ by Theorem~\ref{thm:legammaT-div-leT}
\end{proof}

\begin{corollary} \label{cor:Skolem}
  Suppose that $A$ is an incomplete c.e.\ set, and let $\gamma$ be one of the
  numberings $W_{(-)}$ or  $\varphi_{(-)}$.
  Then $\Skolem(A,B,\gamma)$ if and only if $A\le_T B$.
\end{corollary}

\begin{proof}
  By Example~\ref{ex:divisible} the numberings $W_{(-)}$ and $\varphi_{(-)}$ are
  $\omega$-divisible, hence $2$-divisible, so if $\Skolem(A,B, \gamma)$ then
  $A\le_T B$ by Corollary~\ref{cor:divSkolem}.
  Conversely, if $A\le_T B$ then $\Skolem(A,B,\gamma)$ by
  Theorem~\ref{thm:arslanovparam}.
\end{proof}

In Barendregt and Terwijn~\cite[Question 3.4]{BarendregtTerwijn} the question
was posed whether the uniformly computable existence of fixpoints for a
numbering implies that the numbering is precomplete.  This question remains
open, but we note here that in general the answer is negative for the
\emph{relativized\/} version of this question. Namely, we see from
Theorem~\ref{thm:arslanovparam}, together with
Theorem~\ref{thm:W-phi-comp-equiv}, that even if every $A$-computable family of
$A$-computable functions has a fixpoint modulo $\gamma$, computable uniformly in
$A$, this does not guarantee that $\gamma$ is $A$-precomplete.

\subsection{Comparison of results}

After putting this paper on the arXiv, M. M. Arslanov was kind enough to
send us a yet unpublished but submitted paper (Fixed-point selection
functions) that also contains results about the complexity of computing
fixpoints, some of which are related to the results of this paper.

The two papers use slightly different terminology.  Since we are only concerned with
the behaviour of total ($A$-)computable functions, we have no equivalent of
Arslanov's notion of the existence of a fixed point selection function for a set~$A$,
which is phrased in terms of a universal function relative to~$A$.
Our notion $\Skolem(A, B, \gamma)$ corresponds to the existence of
a $B$-computable fixed point selection function modulo~$\gamma$ for every
$A$-computable function.

The most notable relation between our results is that our
Corollary~\ref{cor:Skolem} generalizes Arslanov's Theorem 7 to the case when $B$
is not comparable to $A$.  The negation of $\Skolem(A, B, \gamma)$ is equivalent
to the existence of a function $h \le_T A$ that has no $B$-computable fixed
point selection function.  Arslanov's result that $h$ can be taken of degree $A$
can be obtained by encoding $A$ into~$h$.

\section*{Acknowledgements}

We thank Andrea Sorbi for several helpful remarks about ceers.


\begin{thebibliography}{99}

\bibitem{AndrewsBadaevSorbi} U. Andrews, S. Badaev, and A. Sorbi,
\textit{A survey on universal computably enumerable equivalence relations},
in: A. Day et al.\ (eds.), Computability and Complexity,
Springer, 2017, 418--451.

\bibitem{Andrewsetal} U. Andrews, S. Lempp, J. S. Miller,
K. M. Ng, L. San Mauro, and A. Sorbi,
\textit{Universal computably enumerable equivalence relations},
Journal of Symbolic Logic 79 (2014) 60--88.

\bibitem{AndrewsSorbi2018} U. Andrews and A. Sorbi,
\textit{Jumps of computably enumerable equivalence relations},
Annals of Pure and Applied Logic 169 (2018) 243--259.

\bibitem{AndrewsSorbi2019} U. Andrews and A. Sorbi, 
\textit{Joins and meets in the structure of ceers}, 
Computability~8 (2019) 193--241.

\bibitem{Arslanov} M. M. Arslanov,
\textit{On some generalizations of the fixed point theorem},
Soviet Mathematics (Izvestiya VUZ. Matematika) 25(5) (1981) 1--10
(English translation).

%\bibitem{Arslanov2020} M. M. Arslanov,
%private email communication, Oktober 2019.

\bibitem{BarendregtTerwijn} H. P. Barendregt and S. A. Terwijn,
\textit{Fixed point theorems for precomplete numberings},
Annals of Pure and Applied Logic 170 (2019) 1151--1161.

\bibitem{BernardiSorbi} C. Bernardi and A. Sorbi,
\textit{Classifying positive equivalence relations},
Journal of Symbolic Logic 48(3) (1983) 529--538.

\bibitem{Ershov} Y. L. Ershov,
\textit{Theorie der {N}umerierungen I}, {Z}eitschrift f\"ur mathematische
{L}ogik und {G}rundlagen der {M}athematik 19 (1973) 289--388.

\bibitem{Ershov2} Y. L. Ershov,
\textit{Theorie der {N}umerierungen II}, {Z}eitschrift f\"ur mathematische
{L}ogik und {G}rundlagen der {M}athematik 21 (1975) 473--584.

\bibitem{Jockuschetal} C. G. Jockusch, jr., M Lerman, R. I. Soare, and R. M. Solovay,
\textit{Recursively enumerable sets modulo iterated jumps and extensions of Arslanov's
completeness criterion}, Journal of Symbolic Logic 54(4) (1989) 1288--1323.

\bibitem{Kleene} S. C. Kleene, \textit{On notation for ordinal numbers},
Journal of Symbolic Logic 3 (1938) 150--155.

\bibitem{Lachlan} A. H. Lachlan,
\textit{A note on positive equivalence relations},
Mathematical Logic Quarterly 33(1) (1987) 43--46.

\bibitem{Moschovakis} Y. N. Moschovakis,
\textit{Kleene's amazing second recursion theorem},
Bulletin of Symbolic Logic 16(2) (2010) 189--239.

\bibitem{Odifreddi} P. Odifreddi, \textit{Classical recursion theory},
Vol.\ 1, Studies in logic and the foundations of mathematics Vol.\ 125,
North-Holland, Amsterdam, 1989.

\bibitem{Selivanov} V. L. Selivanov,
\textit{Index sets of quotient objects of the {P}ost numeration},
Algebra i Logika 27(3) (1988) 343--358
(English translation 1989).

\bibitem{SelivanovWide} V. L. Selivanov,
\textit{Precomplete numberings and functions without fixed points},
Matematicheskie Zametki, 51(1) (1992) 149--155
(English translation 1992).

\bibitem{SelivanovSurvey} V. L. Selivanov,
\textit{Precomplete numberings},
%%% better to write out name of journal instead of abbreviation
Itogi Nauki i Tekhniki, Sovremennaya Matematika i ee Prilozheniya, Tematicheskie Obzory,
Volume 157 (2018) 106-134.

\bibitem{Soare} R. I. Soare, \textit{Recursively enumerable sets and degrees},
Springer-Verlag, 1987.

\bibitem{Terwijn2018} S. A. Terwijn,
\textit{Generalizations of the recursion theorem},
Journal of Symbolic Logic 83(4) (2018) 1683--1690.

\bibitem{Terwijn2020} S. A. Terwijn,
\textit{The noneffectivity of Arslanov's completeness criterion
and related theorems},
Archive for Mathematical Logic 59(5) (2020) 703--713.

\bibitem{Visser} A. Visser,
\textit{Numerations, $\lambda$-calculus, and arithmetic},
in: J. P. Seldin and J. R. Hindley (eds.),
To H. B. Curry: Essays on Combinatory Logic, Lambda Calculus and Formalism,
Academic Press, 1980, 259--284.

\end{thebibliography}
\end{document}